\documentclass{article}
\usepackage{color,cite}
\usepackage{todonotes}

\usepackage{amssymb,amsfonts,amsthm,amsmath,amscd,bbm}
\usepackage{graphics}
\usepackage{graphicx}
\usepackage{supertabular}
\usepackage{tikz}
\usepackage{multirow}

\newtheorem{theorem}{Theorem}
\newtheorem{lemma}{Lemma}

\newtheorem{corollary}{Corollary}
\newtheorem{conjecture}{Conjecture}
\theoremstyle{definition}

\newtheorem{remark}{Remark}

\newtheorem{claim}{Claim}

\def\mv#1{d^{(#1)}}
\def\lex{\mathrel{<_{\text{lex}}}}

\def\G{\mathcal{G}}
\def\B{\mathcal{B}}
\def\R{\mathcal{R}}
\def\M{\mathcal{M}}

\def\ZZ{\mathbb{Z}}

\def\cH{\mathcal{H}}

\title{On Remoteness Functions 
of Exact Slow $k$-NIM with $k+1$ Piles}

\author{
Vladimir Gurvich\thanks{
National Research University Higher School of Economics (HSE), Moscow, Russia; e-mail:
vgurvich@hse.ru and vladimir.gurvich@gmail.com}
\and
Denis Martynov\thanks{
National Research University Higher School of Economics (HSE), Moscow, Russia; e-mail:
taiwong2464@gmail.com}
\and
Vladislav Maximchuk\thanks{
National Research University Higher School of Economics (HSE), Moscow, Russia; e-mail:
vladislavmaximchuk3495@gmail.com}
\and
Michael Vyalyi\thanks
{National Research University Higher School of Economics, Moscow, Russia;
  Institute of Physics and Technology, Dolgorpudnyi, Russia; Federal Research Center ``Computer Science and Control'' of the Russian Academy of Sciences, Moscow, Russia;
 e-mail:
vyalyi@gmail.com}
}

\begin{document}

\maketitle 

{\bf Abstract}
Given integer $n$ and $k$ such that $0 < k \leq n$ 
and  $n$ piles of stones, 
two player alternate turns. 
By one move it is allowed 
to choose any  $k$  piles and remove  
exactly one stone from each.
The player who has to move but cannot is the loser. 
Cases  $k=1$  and $k = n$  are  trivial. 
For $k=2$  the game was solved for  $n \leq 6$. 
For $n \leq 4$  the Sprague-Grundy function was 
efficiently computed 
(for both the normal and mis\`ere versions). 
For $n = 5,6$  a polynomial algorithm 
computing P-positions was obtained. 
Here we consider the case $2 \leq k = n-1$ 
and compute Smith's remoteness function,  
whose even values define the P-positions. 
In fact, an optimal move 
is always defined by the following simple rule:
\begin{itemize}
\item{}
if all piles are odd, 
keep a largest one and reduce all other;
\item{} if there exist even piles, 
keep a smallest one of them 
and reduce all other. 
\end{itemize}

Such strategy is optimal for both players, moreover,  
it allows  to win as fast as possible 
from an N-position and 
to resist as long as possible from a P-position. 

\medskip 


\medskip 

{\bf AMS subjects:  91A05, 91A46, 91A68} 

\section{Introduction} 
\label{s0}
We assume that the reader is familiar 
with basic concepts of impartial game theory; 
see e.g., \cite{ANW07,BCG01-04} for an introduction. 

\subsection{Exact Slow NIM} 
\label{ss00}
Game Exact Slow NIM  
was introduced in \cite{GH15} as follows: 
Given two integers  $n$ and $k$  
such that  $0 < k \leq n$  and  
$n$  piles containing  $x_1, \dots, x_n$  stones.  
By one move it is allowed to reduce any  $k$  piles 
by exactly one stone each. 
Two players alternate turns. 
One who has to move but cannot is the loser.  
In \cite{GH15}, this game was denoted  NIM$^1_=(n,k)$.
Here we will simplify this notation to  NIM$(n,k)$.
We start with three simple properties of this game. 

\smallskip

Obviously, value    
$x_1 + \dots + x_n \; \mod k$  
is an invariant of game NIM$(n,k)$.  
Hence, NIM$(n,k)$  is split into  
$k$  independent subgames, 
between which there are no moves. 

\smallskip

A position $x = (x_1, \dots, x_n)$
will be called {\em even} or {\em odd}  
if all its coordinates are even or odd, respectively. 

\medskip 

Every even position of 
game NIM$(n,k)$  is a P-position, 
because each move of the first player 
can be repeated by the opponent. 
Respectively, every position  $x$  
with exactly  $k$  odd coordinates 
is an N-position, because there is a move from  $x$ 
to a P-position. 

\begin{remark} 
Note that both above claims 
can be obviously extended from game NIM$(n,k)$  
to the slow NIM to arbitrary hypergraph \cite{GH15}. 
Given a hypergraph  
$\cH \subseteq 2^{[n]} \setminus \{\emptyset\}$ 
on the ground set $[n] = \{1, \dots, n\}$,
and a nonnegative vector  $x = (x_i \mid i \in [n])$, 
the {\em hypergraph slow game NIM$_\cH$}  
is defined as follows: 
By one move, a player can choose any hyperedge 
$H \in \cH$ such that  $x_i > 0$  for every $i \in H$ 
and reduce by 1  each $x_i$  for  $i \in H$.  
Then, game NIM$(n,k)$  corresponds to 
the hypergraph slow NIM$_{\cH}$   
with   $\cH = \{H \mid |H| = k\}$. 
In any  hypergraph slow NIM$_{\cH}$,  every even $x$  is a P-position 
and, hence,  $x$  is an N-position when  
$x_i$  is odd if and only if  $i \in H$
for some hyperedge $H \in \cH$.
\end{remark} 

Game NIM$(n,k)$  is trivial if  $k = 1$  or  $k = n$. 
In the first case it ends after $x_1 + \dots + x_n$ moves  
and in the second one---after  $\min(x_1, \dots, x_n)$ 
moves. In both cases nothing depends on players' skills. 
All other cases are more complicated. 

The game was solved for  $k=2$  and  $n \leq 6$.
In \cite{GHHC20}, an explicit formula 
for the Sprague-Grundy 
(SG) function was found for  $n \leq 4$
(for both, the normal and mis\`ere versions of the game). 
This formula allows us to compute  the SG function in linear time. 
Then, in \cite{CGKPV21} the P-positions were found for  $n \leq 6$.  
For the subgame with even  $x_1 + \dots + x_n$ 
a simple formula for the P-positions was obtained. 
It allows for a given position, in linear time,  
to verify if it is a P-position and, if not, 
to find a move from it to a P-position.  
The subgame with odd  $x_1 + \dots + x_n$ is  more difficult. 
Still, a more sophisticated  formula for the P-positions was found and it gives a linear time algorithm for recognizing P-positions. 

\subsection{Case $n = k+1$}
\label{ss01}
In this paper we solve the game in case $n = k+1$. 
In every position an optimal move 
is always defined by the following simple rule: 
\begin{itemize}
\item[] if all piles are odd, 
keep a largest one and reduce all other;
\item[] if there exist even piles, 
keep a smallest one of them 
and reduce all other. 
\end{itemize}

We will call this strategy the {\em M-rule} and 
the corresponding moves by the {\em M-moves}. 
Obviously, in every position 
there exists a unique M-move, 
up to re-numeration of piles. 
We will show that M-rule solves the game, moreover, 
it allows to win as fast as possible in an N-position and 
to resist as long as possible in a P-position. 

\smallskip 

Given a position  $x = (x_1, \ldots, x_n)$, 
assume that both players follow the M-rule and  
denote by $\M(x)$  the number of moves 
from  $x$  to a terminal position. 
We will prove that  $\M = \R$, where 
$\R$  is the so-called remoteness function 
considered in the next subsection.

\smallskip  

Let us note that in some positions 
the M-move may be a unique winning move. 
For example, in positions 
(1,1,2)  and  (1,1,3)  of NIM(3,2)  
the M-move  reduces  1,1  to 0,0, 
and wins immediately, while moves 
to  (0,1,1)  and  (0,1,2)  are losing. 

\smallskip 

Let us note 
that there exists no M-move to an odd position. 
Yet, the latter may be a P- or N-position. 
For example, in NIM(3,2), position 
(1,1,1)  is an N-position, 
while  (3,3,3)  is a P-position, $\M(3,3,3) = 4$. 
The corresponding sequence of M-moves is: 
$(3,3,3) \to (2,2,3) \to (1,2,2) \to (0,1,2) \to (0,0,1).$ 

\subsection{Smith's Remoteness Function} 
\label{ss02}
In 1966 Smith \cite{Smi66} 
introduced a {\em remoteness} function $\R$ 
for impartial games by the following algorithm. 

Consider an impartial game 
given by a finite acyclic directed graph 
(digraph)  $G$.  
Set  $\R = 0$  for all terminal positions 
of  $G$  and  
$\R(x) = 1$  if and only if 
there is a move from $x$ to a terminal position.  
Delete all labeled positions from  $G$  and 
repeat the above procedure increasing  $\R$  by 2,  
that is, assign 2 and 3  instead of  0 and 1; etc. 

This algorithm was considered 
as early as in 1901 by Bouton \cite{Bou901},  
but only for a special graphs 
corresponding to the game of NIM. 
In 1944 this algorithm was extended to 
arbitrary digraphs by von Neumann and Morgenstern \cite{NM44}. 
In graph theory, the set of positions with even  $\R$  
they called the {\em kernel} of a digraph; 
in theory of impartial games, this set 
is referred to as the set of  P-positions. 

Function $\R$ has also the following, stronger, properties: 
Position $x$  is a P-position if and only if  $\R(x)$  is even. 
In this case the player making a move from  $x$  cannot win, yet,   
can resist at least $\R(x)$  moves, but not longer. 
Position  $x$  is an N-position if and only if  $\R(x)$  is odd.  
In this case the player making a move from  $x$ wins in at most 
$\R(x)$ moves, but not faster. 
In both cases the player making a move from   $x$  should  
reduce $\R$  by one, which is always possible
unless position  $x$  is terminal,  
in which case  $\R(x) = 0$.

\subsection{Sprague-Grundy and Smith's Theories}  
\label{ss03}
Given $n$ impartial games 
$\Gamma_1, \dots, \Gamma_n$, 
by one move a player chooses one of them 
and makes a move in it. 
The player who has to move but cannot is the looser. 
The obtained game  
$\Gamma = \Gamma_1 \oplus \dots \oplus \Gamma_n$  
is called the {\em disjunctive compound}. 
For example, NIM with $n$ piles 
is the disjunctive compound of $n$  one-pile NIMs. 
The SG function $\G$ of $\Gamma$ is uniquely determined 
by the SG functions of $n$ compound games by formula 
$\G(\Gamma) = \G(\Gamma_1) \oplus \dots \oplus \G(\Gamma_n)$ 
where $\oplus$  is the so-called NIM-sum. 
These results were obtained by Bouton 
\cite{Bou901}  for the special case of NIM and 
then extended to arbitrary impartial games 
by Sprague \cite{Spr36} and Grundy \cite{Gru39}. 

\medskip

Now suppose that   
by one move a player makes a move 
in each of the $n$ compound games, 
rather than in one of them. 
Again, the player who has to move but cannot is the looser.
The obtained game  
$\Gamma = \Gamma_1 \wedge \dots \wedge \Gamma_n$  
is called the {\em conjunctive compound}. 
Remoteness function $\R$ of $\Gamma$ is uniquely determined 
by the remoteness functions of $n$ compound games by formula 
$\R(\Gamma) = \min(\R(\Gamma_1), \dots, \R(\Gamma_n))$. 
This result was obtained in 1966 by Smith \cite{Smi66}. 

\subsection{On \boldmath $m$-critical positions of game NIM$(n,n-1)$}
\label{ss04}

The relation  $x \succeq y$  is defined standardly as follows. 
Order, monotone increasingly, the coordinates of 
$n$-vectors $x$ and $y$  and denote the obtained 
$n$-vectors by $x^*$ and $y^*$. 
We write $x \succeq y$  if 
$x^*_i \geq  y^*_i$  for all  $i \in \{1, \dots, n\}$ 
and  $x \succ y$  
if, in addition, $x^* \neq y^*$, 
or in other words, 
if at least one of these 
$n$ inequalities is strict. 

In the set of positions $x$ with fixed value $\M(x)$ 
we distinguish minimal with respect 
to the relation $\succeq$. 
Given a nonnegative integer  $m$,    a position $x = (x_1, \dots x_n)$  
of game NIM$(n,k)$  with $n = k+1$  is called  $m$-critical if  $\M(x) = m$  and  $x \succ y$ fails for any position $y$  with  $\M(y) = m$. 

By this definition, for any position $y$   
there exists a unique $m = m(y)$  
such that  $y \succeq x$  holds 
for some $m$-critical $x$  and   
$y \succeq x$  fails for 
any $(m+1)$-critical  $x$.

\smallskip 

We will characterize  $m$-critical positions 
of game  NIM$(k+1,k)$  
as follows. 

\begin{theorem}\label{th:critical}
Position $x = (x_1, \dots, x_n)$  
of game NIM$(n,k)$  with $n = k+1$  
is $m$-critical if and only if 
one of the following two cases holds: 

\bigskip 

\textup{(A)}  $\;\; x_1 + \dots + x_n = km$; 

\smallskip 

$\max(x_1, \dots, x_n) \leq m$;   

\smallskip 

if $m$ is even then  $x$  is even 
(that is, all $n$ coordinates of $x$  are even);  

\smallskip 

if $m$ is odd then exactly one of $n$ coordinates 
of  $x$  is even. 

\bigskip 

\textup{(B)} $\;\; x_1 + \dots + x_n = km + k - 1$;

\smallskip 

$\max(x_1, \dots, x_n) < m$;   

\smallskip 

$m$ is even, $x$ is odd. 
\end{theorem}

\medskip 

The proof will be given in Section~\ref{sec:proofs12}. 

\medskip 

Note that case (B)  provides a large family 
of odd P-positions. 
Note also that there are odd P-positions beyond 
this family too. 
For example, $\M(3,5,5) = 6$ in NIM$(3,2)$. 
The corresponding sequence of  M-moves is: 
$(3,5,5) \to (2,4,5) \to (2,3,4) \to (2,2,3) 
\to (1,2,2) \to (0,1,2) \to (0,0,1).$

\medskip 

Recall that, by definition, 
an M-move cannot lead to an odd position.  

\bigskip 

The following statement easily result   
from the above theorem.

\begin{corollary}
\label{cor1}
(i) Let  $x$ be an $m$-critical position with an even $m$. 
Then, $x$ is either odd \textup(Case \textup{(B)}\textup) or even \textup(Case \textup{(A)}\textup).    
In the latter case, at least one of even coordinates of  $x$ 
is strictly less than  $m$. 
A move from  $x$  leads 
to an odd $(m-1)$-critical position 
provided this move reduces 
all (maximal) coordinates of  $x$ 
that are equal to  $m$. 
In particular, an M-move has such property. 

\medskip

(ii) For an $m$-critical position $x$ with an odd $m$   
there exists exactly one move 
leading to an even $(m-1)$-critical position. 
This move keeps the (unique) even coordinate 
and reduces all other; 
furthermore, this move is a unique M-move. 
\end{corollary}

\proof 
(i) 
It is enough to note that 
all $n=k+1$  coordinates 
cannot be equal to  $m$, because 
in this case their sum would be  $(k+1)m$ 
rather than  $km$, 
as required by Theorem \ref{th:critical}.  
All other statements are immediate 
from the definitions and Theorem \ref{th:critical}. 

\medskip

(ii)
If the largest coordinate $x_i$ of $x$  is even, 
then  $x_i$  is unique and  $x_i < m$, 
since  $m$  is odd. 
Hence, by Theorem \ref{th:critical}, 
the M-move is unique and leads to 
an $(m-1)$-critical position. 
All other statements are immediate 
from the definitions and Theorem \ref{th:critical}. 
\qed

\bigskip 

Due to Theorem~\ref{th:main} below, a suitable generalization 
of $m$-critical positions on arbitrary game uses 
the remoteness function  $\R(x)$ instead of $\M(x)$. 
Characterizing  $m$-critical positions of game  NIM$(n,k)$ 
remains an open problem. 

\begin{conjecture} 
For any integer $k,n$ and $m$ 
such that  $0 < k < n$  and $0 \leq m$, inequalities 

\medskip 

$km \leq (x_1 + \dots + x_n) < k(m+1)\;$  
and 
$\; \max(x_1, \dots, x_n) \leq m$  

\medskip 
\noindent 
hold for any $m$-critical position   
$x =(x_1, \dots, x_n)$  of game NIM$(n,k)$. 
\end{conjecture} 

It seems difficult to characterize 
the $m$-critical positions for 
arbitrary $k,n$ and $m$. 
Consider, for example,  
$k = 3, n = 5$ and $m = 8$. 
Computations show that 
$(1,3,7,7,7), (1,5,5,7,7), (3,3,5,7,7)$, and  $(5,5,5,5,5)$ 
are $m$-critical positions of game  NIM$(5,3)$, 
while  $(3,5,5,5,7)$  is not; instead  $(3,5,5,6,7)$  is. 

\subsection{Main Results} 

Our main results are 
the following two properties of game NIM$(n,n-1)$. 

\begin{theorem}
\label{th:main}
Equalities  $m(x) = \M(x) = \R(x)$ 
hold for any position  $x$. 
\end{theorem}

\begin{theorem}\label{th:complexity}
Function $\R$ can be computed in polynomial time, 
even if $n$  is a part of the input and integers are presented in binary.
\end{theorem}

Hence, P-positions can be recognized in polynomial time.

\subsection{On P-positions of game \boldmath NIM$(n,n-1)$} 
\label{ss05}
Recall that P-positions of an impartial game 
can be characterized in two ways, in other words,  
the following three statements are equivalent: 
$x$ is a P-position, $\G(x) = 0$, and $\R(x)$ is even. 
Hence, $x$  is a P-position of NIM$(n,n-1)$
if and only if the number $m(x) = \M(x) = \R(x)$  is even. 
Using the above theorems, we will obtain a 
polynomial algorithm verifying  
if  $x$  is a  P-position and if not 
computing a move from  $x$  to a P-position.   

Interestingly, 
we do not need this algorithm to play NIM$(n,n-1)$. 
It is much simpler just to follow the M-rule 
(see subsection \ref{ss00})---a~peculiar situation in impartial game theory, 
which somewhat downgrade the role of P-positions. 

\smallskip

For  NIM$(4,3)$ some pretty complicated formulas 
for P-positions were confirmed by computations; see Appendix. 
Yet, we  were not able to prove these formulas. 
Furthermore, they differ in three cases:
$x_1 + x_2 + x_3 + x_4 \equiv 0,1$, or $2 \; (\mod 3)$.
Game NIM$(k+1,k)$  is split into  $k$  disjoint subgames,  
which makes the general case analysis even more difficult. 
Interestingly, the $m$-critical positions, 
unlike the P-positions, are defined 
by the same formulas for all $k$  subgames.  

\subsection{Related works}
The ancient game of NIM was solved 
in 1901 by Bouton \cite{Bou901}. 

In 1907 Wythoff \cite{Wyt907} introduced and 
solved modified NIM 
with two piles of stones, where 
by one move a player can either reduce one pile, 
or both, but then by the same number of stones. 
Wythoff found the P-positions.
The game and its generalizations were studied 
in many works; see, for example, 
[2, 13, 14, 16–18, 22, 37, 39]. 

The SG function of this game is still unknown 
despite many efforts; 
see, for example, \cite{BF90,Con76,Niv09}. 
More precisely, it is unknown 
whether SG value $\G(x)$  can be computed 
in polynomial time for a given position $x$. 
In contrast, the remoteness function $\R(x)$ 
can be computed in polynomial time 
for the Wythoff game and several its generalizations 
\cite{preprint}. 

In 1910 Moore introduced and solved $k$-NIM. 
By one move a player can reduce 
at most $k$  piles, and at least 1.  
Moore found the P-positions. 
The SG-function is known only in the case 
$n = k+1$; see Jenkins and Mayberry \cite{JM80}. 
Recently, their result was considerably strengthened 
and extended from hypergraph  $k$-NIM with $n = k+1$  piles 
to the so-called 
hypergraph NIM on arbitrary JM hypergraphs; 
see  \cite{BGHM15,BGHMM15,BGHMM17,BGHMM18} 
for details. 

For Moore's game, there is a simple formula 
for $\R(x)$  when  $x$  is a P-position; 
in contrast, $\R(x)$ is NP-hard to compute 
when $x$  is an N-position;
see \cite{preprint}. 

Both SG and remoteness functions 
can be computed in polynomial time 
for the game Euclid \cite{CD69}; 
see also \cite{Fra05,Gur07,Len03,Len04}. 
A very simple and elegant formula  
for the SG function was given 
in 2006 by Nivasch \cite{Niv06}.
A polynomial algorithm computing 
the remoteness function was recently suggested 
in \cite{preprint}.

\section{Proofs of Theorems~\ref{th:critical} and~\ref{th:main}}\label{sec:proofs12}

Here we always assume that coordinates in a position 
$x$ are ordered in the non-decreasing order and after a move coordinates are re-ordered.

In the case $k=1$, the theorems  trivially hold. We assume in this section that $k>1$.

\subsection{Function \boldmath $\B(x)$}

We introduce an auxiliary function $\B(x)$ using conditions (A) from Theorem~\ref{th:critical}. Let $|x| = \sum_{i=1}^{k+1}x_i$ be the $\ell_1$-norm of $x$. A~position $z=(z_1,\dots,z_{k+1})$ is called \emph{basic}, if 
\begin{align}
  &|z| = kb(z),\label{m-def}\\
  &z_i\leq b(z),\label{m-bound}\\
  &\text{\parbox{7cm}{for even $b(z)$ all $z_i$ are even,\\ and for odd $b(z)$ exactly one $z_i$ is even.}}\label{m-parity}
\end{align}
Then
\begin{equation}\label{eq:G-def}
\B(x) = \max (b(z) : x\succeq z,\ z\ \text{is basic}).
\end{equation}
It is clear that $\B(z) = b(z)$ for every basic position $z$. In general,  for a basic position~$z$ such that $b(z) = \B(x)$,  we say that $z$ \emph{supports} $x$ (or $x$ is supported by $z$). 
If $\B(x)$ is even, then $x$ is called $B$-even. Otherwise, it is called $B$-odd.

In general, coordinates in a basic position are not ordered. But it does not matter due to the following observation. 
\begin{claim}\label{ordered}
  If   $z\preceq x$ for a basic $z$, then there exists a
  basic position $z^*$ such that $z^*\preceq x$,  and $b(z^*) = b(z)$, and   coordinates of   $z^*$ are ordered in the non-decreasing order.
\end{claim}
\begin{proof} All the conditions~\eqref{m-def}, \eqref{m-bound} and \eqref{m-parity} are invariant  under  permutations of coordinates. If $x$ is ordered and $z_i>z_j$ for $i<j$, then $x_j\geq x_i\geq z_i$ and $x_i\geq z_i> z_j$. Thus, transposition of $z_i$ and $z_j$ gives a basic position. Repeating the transpositions, we come to an ordered basic $z^*$ such that  $z^*\preceq x$.
\end{proof}

For ordered basic positions we have a useful inequality.

\begin{claim}\label{12bound}
 If  $z$ is an ordered basic position, then $z_1+z_2 \geq b(z)$.
\end{claim}
\begin{proof}
The arithmetic mean of $z_1+z_2$, $z_i$, $i\geq3$, is exactly $b(z)$. The claim 
follows from  $z_i\leq b(z)$ (the bound condition~\eqref{m-bound}).
\end{proof}

From this claim we immediately get a useful corollary.

\begin{corollary}\label{z1-fact}
  If  $z$ is an ordered basic position, then $z_1\geq0$. If $z_1=0$ then $z=\big(0, b(z),\dots, b(z)\big)$.
\end{corollary}

A move in NIM$(k+1,k)$ preserves exactly one coordinate. We use this coordinate to indicate the move. Thus, a move~$i$ from a~position $x$ leads to the position $x-\mv{i}$, where $\mv{i}_j = 1-\delta_{ij}$ ($\delta_{ij} $ is the Kronecker symbol). 

Let $x\to x'$ be a move in NIM$(k+1,k)$. To prove main results, we relate  $\B(x)$ and $\B(x')$ in  claims below.  

\begin{claim}\label{terminal}
  If $\B(x)=0$ then $x$ is a terminal position in \textup{NIM}$(k+1,k)$.
\end{claim}
\begin{proof}
  If $x_2>0$ then $z=\mv{1}=(0,1,\dots,1)\preceq x$, $b(z)=1$, the parity condition~\eqref{m-parity} and the bound condition~\eqref{m-bound} hold. Thus $z$ is basic and $\B(x)\geq 1$.

  If $x_1=x_2=0$ then $x$ is terminal.
\end{proof}

\begin{claim}\label{Bnonincrease}
  If $x\to x'$ is a move in \textup{NIM}$(k+1,k)$ then $\B(x')\leq \B(x)$.
\end{claim}
\begin{proof}
  Let $z'$ be a basic position supporting $x'$. Then $z'\preceq x'\preceq x$ and $\B(x)\geq \B(x')$.
\end{proof}

\begin{claim}\label{DeltaGodd}
  If $x\to x'$ is a move in \textup{NIM}$(k+1,k)$ from a $B$-odd position~$x$, then  $\B(x') \geq \B(x)-2$.
\end{claim}
\begin{proof}
  Let $z$ be an ordered basic position that supports $x$ and $z_p$ is even (other coordinates in $z$ are odd). If $b(z)\leq 2$ then the claim trivially holds. For the rest of the proof assume that $b(z) \geq 3$. It implies $z_2\geq2$. If $p=2$, it follows from Claim~\ref{terminal}. If $p\ne2$, it follows from Claim~\ref{12bound}: $z_1+z_2\geq b(z)\geq 3$ and $z_1\leq z_2$ imply that   $z_2\geq 2$.
  
  Define $z'$ as follows.

  If $p>1$ then $z'_1=z_1-1$, $z'_p=z_p-1 $,   $z'_i= z_i-2$ for $i\notin \{1,p\}$.

  If $p=1$ and $z_1>0$ then  $z'_1=z_1-1$, $z'_2=z_2-1 $,   $z'_i= z_i-2$ for $i\geq3$.

In both cases, $b(z') =b(z)-2 $, the parity condition holds for $z'$, as well as the bound condition. Thus, $z'$ is basic.
  
  If $p=1$ and $z_1=0$ then $z=\big(0,b(z),\dots, b(z)\big)$ due to Corollary~\ref{z1-fact}. In this case $z'_1 = 0$, $z'_i = z_i-2$ for $i\geq2$.
It is clear that $z'$ is basic and $b(z')= b(z)-2$.
  
In all three cases $0\preceq z'\preceq x'$. Thus $\B(x')\geq b(z') = b(z)-2=\B(x)-2$.
\end{proof}

\begin{claim}\label{ToGeven}
  If $x\to x'$ is a move in \textup{NIM}$(k+1,k)$ to a $B$-even position~$x'$, then  $\B(x')< \B(x)$.
\end{claim}
\begin{proof}
  If $x-x' = \mv{i}$ and  $z'$ supports $x'$ then $z=z'+\mv{i}\preceq x$, and $b(z) = b(z')+1$, and $z_i\leq b(z) $,  and the parity condition~\eqref{m-parity} holds for $z$: since  all coordinates are even in $z'$, exactly one coordinate is even in $z$. Thus $\B(x)\geq \B(z) = \B(z')+1 = \B(x')+1$.
\end{proof}

\begin{lemma}\label{GevenGeven}
  There are no moves from  $B$-even positions to  $B$-even positions in \textup{NIM}$(k+1,k)$.
\end{lemma}
\begin{proof}
  Suppose for the sake of contradiction that $x\to x' = x-\mv{i}$ is a move and $x$, $x'$ are $B$-even.

  Let $b = \B(x)$ and $z$ be an ordered basic position supporting $x$. Consider several cases. It follows from  Claim~\ref{ToGeven} that $\B(x')\leq b-2$.

  The case $z_1=0$. In this case  $z= \big(0,b(z),\dots, b(z)\big)$, due to Corollary~\ref{z1-fact}, and $x'_j\geq x_j-1\geq z_j-1$ for $j>1$, and $x'_1\geq 0=z_1$.
 Hence, $z'=z-\mv{1}$ is basic and $z' \preceq x'$. It implies that $\B(x')\geq b(z') = b(z)-1= \B(x)-1$.  We come to a~contradiction:  $b-2\geq \B(x')\geq b(z')= b-1$.

   The case $z_1\geq1$ and $z_i< b$. Let us prove that $z' = z-\mv{i}$ is a basic position and $z'\preceq x'$. The latter is clear from the definitions. Since $b(z')=b(z)-1$ is odd, the parity condition~\eqref{m-parity} holds. The bound condition~\eqref{m-bound} is due to the assumption~$z_i< b$. Again, we come  to a~contradiction:  $b-2\geq \B(x')\geq b(z')= b-1$.

   The case $z_1\geq1$, and $z_i= b$, and there exists $j< i$ such that $z_j<\min(x_j,b)$. Now $z' = z-\mv{i}$ violates the bound condition~\eqref{m-bound}, but the conditions~\eqref{m-def},~\eqref{m-parity} hold by the same reasons as before. That is, $b(z')$ is odd and the only even position in $z'$ is~$i$. Take $z'' = z-\mv{j}$. In other words, 
\[
z''_r = \left\{
\begin{aligned}
  &z'_r = z_r-1, &&\text{if}\ r\ne i, \ r\ne j,\\
  &z'_j+1 = z_j, &&\text{if}\ r=j,\\
  &z'_i-1 =z_i-1, &&\text{if}\ r=i.
\end{aligned}\right.
\]
The conditions~\eqref{m-def},~\eqref{m-parity} hold for $z''$  and $b(z'')=b-1$. Since $z_j<x_j$, we get $z''\preceq x'$. From $z_j<b$ we get the bound condition~\eqref{m-bound}. Again, we come  to a~contradiction:  $b-2\geq \B(x')\geq b(z'')= b-1$.

The remaining case $z_1\geq1$, and $z_i= b$, and for all $j$ either $z_j=x_j$ or  $z_j=b$. Since $z$ is ordered, the former holds for $j<t$ and the latter holds for $j\geq t$. W.l.o.g. we assume that $z_j<x_j$ for $j\geq t$. Note that $t=k+2$ is possible (in this case $z = x$ and, w.l.o.g., $i=k+1$). Take a $B$-even ordered basic position $z'$ supporting $x'$ (in particular, $z'\preceq x'$). It follows from Claim~\ref{ToGeven} that $b(z')\leq b-2$. All $z'_j$ are even since $x'$ is $B$-even. Therefore $z'_j< x'_j = x_j-1 = z_j-1$ for $j<t$, since the former is even and the latter is odd and $z'_j \leq b-2 < x_j-2 = x'_j-1$ for $j\geq t$. Thus $z''=z'+\mv{j}\preceq x'$ for all $j$ and $z''$ is a basic position such that $b(z'') = b(z')+1 $. It implies that $\B(x')\geq b(z')+1 = \B(x')+1$, a~contradiction.
\end{proof}

Positions satisfying the condition (B) in Theorem~\ref{th:critical} are called \emph{exceptional}. We establish several facts about exceptional positions.

\begin{claim}\label{exceptional0}
  $\B(x)<m(x)$ for an exceptional position $x$.
\end{claim}
\begin{proof}
  Let $z$ be a basic position that supports $x$. Then
\[
k\B(x) = |z| \leq |x| = km(x) +k-1 = k\big(m(x)+1\big) -1,
\]
thus $\B(x)\leq m(x)$.

Suppose, for the sake of contradiction, that $\B(x)=m(x)\equiv0\pmod2$. It implies $z_i< x_i$ for  all $i$ (since the former is even and the latter is odd). Therefore
\[
k\B(x) = |z|\leq \sum_{i} (x_i-1) = km(x) +k-1 -k-1= km(x) -2,
\]
a contradiction.
\end{proof}

\begin{claim}\label{exceptional1}
  If $x\succeq z$, and $x\ne z$, and $z$ is exceptional then $\B(x)> m(z)$.
\end{claim}
\begin{proof}
  Let $x_i>z_i$. Take $z' = z+e^{(i)}$, where $e^{(i)}_j = \delta_{ij}$. Then
\[
|z'|= 1+| z| = 1+ km(z) +k-1 = k(m(z)+1),
\]
$z'_i$ is even, while $z'_j$ is odd for all $j\ne i$,  and $z'_j<m(z)+1$. Thus $z$ is basic and $z'\preceq x$. So, $\B(x)\geq m(z)+1$.
\end{proof}

\begin{corollary}\label{exeptional2}
$\B(x)= m(x')+1$  for any move $x\to x'$ to an exceptional position~$x'$.
\end{corollary}

\begin{proof}
   Claim~\ref{exceptional1} implies $\B(x)\geq m(x')+1$. On the other hand,
  \[
  |x|=|x'|+k =  km(x')+k-1 +k < k\big(m(x')+2\big).
  \]
  Thus, $\B(x)\leq m(x')+1$.
\end{proof}

\begin{claim}\label{exceptional3}
  For any move $x\to x'$ from an exceptional position, $\B(x)=\B(x')=m(x)-1$  is odd. 
\end{claim}

\begin{proof}
  From Claim~\ref{exceptional0} we get $\B(x)<m(x)$. 
Let $x' = x-\mv{i}$ and $p=1$, if $i\ne 1$, $p=k+1$, if $i=1$. Define $z'$   as follows
\[
z'_j = \left\{\begin{aligned}
  &x_i <m(x), &&\text{if}\ j=i,\\
&x_p-1, &&\text{if}\ j=p,\\
&x_j-2, &&\text{otherwise.}
\end{aligned}\right.
\]
It follows from the definition that $z'\preceq x'\preceq x$, and
\[
| z'| = -1-2(k-1) +|x |= km(x) +k-1 +1 -2k = k\big(m(x)-1\big),
\]
and the only even coordinate in $z'$ is $p$. Since $m(x)$ is even, $z'$ is basic and $m(x) > \B(x)\geq m(x)-1$, $\B(x')\geq m(x)-1$. Note that 
\[
|x'|=\big(km(x) +k-1\big)-k = km(x)-1,
\]
thus $m(x)> \B(x')$.
Therefore, $\B(x) =\B(x') = m(x)-1 $.
\end{proof}

\subsection{Relation to \boldmath $\M(x) $ and $\R(x)$}

At first, we prove several facts about changes of $\B(x)$ on M-moves. We will use the following notation.
Let $e(x)$ be the maximal index such that $x_{e(x)}$ is the smallest even integer among $x_i$ if there are even $x_i$. Otherwise, if all $x_i$ are odd, then $e(x)=k+1$. The M-move at $x$ is $x\to x' = x-\mv{e(x)}$. Note that  defined by this rule $x'$ is ordered (assuming, as everywhere, that $x$ is ordered).

\begin{claim}\label{FromGeven}
   If $x\to x'= x-\mv{e(x)}$ is an M-move in \textup{NIM}$(k+1,k)$ from a $B$-even position~$x$, then  $\B(x') = \B(x)-1$.
\end{claim}
\begin{proof}
 $\B(x)>0$ due to Claim~\ref{terminal}. By Lemma~\ref{GevenGeven}, $x'$ is $B$-odd. Thus, due to Claim~\ref{Bnonincrease}, it is enough to prove  $\B(x')\geq \B(x)-1$. 

  Let $z$ be an ordered basic position supporting $x$. Take $z'=z-\mv{e(x)}$. It is clear that $z'\preceq x'$ and $b(z')= b(z)-1$. The parity condition~\eqref{m-parity} holds for~$z'$: since  all coordinates are even in~$z$, exactly one coordinate is even in~$z'$.
If the bound condition~\eqref{m-bound} also  holds, we have $z'\preceq x'$, $z'$ is basic. Thus $\B(x')\geq b(z') =b(z)-1 =  \B(x)-1$.
  
If the bound condition~\eqref{m-bound} is violated, then $z_{e(x)} =b(z)$. Choose the minimal $i$ such that $z_i=z_{i+1}=\dots=z_{k+1}=b(z)$. Note that $i>1$, otherwise $|z|=(k+1)b(z)\ne kb(z)$ and~\eqref{m-def} is violated. Since $e(x)\geq i$, $x_j$ is odd for $j<i$. In particular,  $x_{i-1}$ is odd and $z_{i-1}$ is even. Thus $z_{i-1}< x_i$. Therefore  $z'' = z- \mv{i-1} $ satisfies $z''\preceq x'$ and all the conditions~(\ref{m-def}--\ref{m-parity}). 

  Since $b(z'') = b(z)-1$, we get  $\B(x')\geq b(z'') =b(z)-1= \B(z)-1 = \B(x)-1$.
\end{proof}

\begin{claim}\label{M-moveFromCritical}
  An M-move from a basic position $z$ leads to a basic position $z'$ such that  $\B(z')=\B(z)-1$.
\end{claim}
\begin{proof}
The sum of coordinates decreases by $k$ at a move. Thus~\eqref{m-def} holds for $z'$ and $b(z') = b(z)-1$. 
  
  Let us prove that $z_{e(z)}<b(z)$, it guarantees that~\eqref{m-bound} holds for $z'$. If $b(z)$ is even, then  all coordinates of $z$ are even and the smallest coordinate is strictly less than $b(z)$ \big(otherwise,  $|z| = (k+1)b(z)> kb(z)$\big).
  If $b(z)$ is odd, then exactly one coordinate is even and it is less than $b(z)$.

  The parity condition is verified in a similar way. 
  If $b(z)$ is even then  there is exactly one even coordinate in $z'$.
  If $b(z)$ is odd then the M-move  $z\to z'$ keeps the even coordinate and decreases the rest of coordinates by 1. So, all coordinates in $z'$  are even.
  
  Thus $z'$ is basic,  and $\B(z')=b(z')  =b(z)-1= \B(z)-1$.
\end{proof}

\begin{claim}\label{FromGodd}
   If $x\to x'$ is an M-move in \textup{NIM}$(k+1,k)$ from a $B$-odd non-exceptional position~$x$. Then $\B(x') = \B(x)-1$. 
\end{claim}
\begin{proof}
  The equality follows from two inequalities $\B(x') \geq \B(x)-1$ and $\B(x') \leq \B(x)-1$.

\smallskip
  
  To prove $\B(x') \geq \B(x)-1$, take an ordered basic position $z$ that supports $x$ and make an M-move $p = e(z)$ from $z$ to $z'$. Note that $z'$ is basic due to Claim~\ref{M-moveFromCritical} and $z_p=z'_{p}\leq b(z') = b(z) -1$, since $z_p$ is even and $b(z)$ is odd.

  If   $x_p>z_p$ then $x'_p\geq z_p$ and $z'\preceq x'$. Thus $\B(x')\geq b(z') = b(z)-1 = \B(x)-1$ and we are done.

If  $e(x)=p$ then, again,  $z'\preceq x'$ and we are done.
  
Assume now that $x_p=z_p$ (so there are even coordinates in $x$) and $e(x)\ne p$. Therefore $z_{e(x)}< x_{e(x)}$, since the former is odd (the even coordinate in $z$ is $p$) and the latter is even.

Note also that $z_p=x_p>0$. Otherwise, $z_p=x_p=0$ and, since $z$ is ordered,   $p=1$ and  $x_p=z_p$ is the smallest even coordinate in $x$ which implies $e(x) = p =1$, a~contradiction to the above assumptions.

Now  define $z''$ by the rules
\[
z''_j = \left\{
\begin{aligned}
  &z'_j, &&\text{if}\ j\ne e(x), \ j\ne p,\\
  &z_{e(x)}+1, &&\text{if}\ j= e(x),\\
  &z_p-2, &&\text{if}\ j= p.
\end{aligned}\right.
\]
We are going to prove that $z''$ is basic and $b(z'') = b(z)-1$.
All coordinates in $z''$ are even. Note that $|z'|=|z''|= kb(z)-k $, since \[(z_{e(x)}-1) + z_p = (z_{e(x)}+1)+(z_p-2).\] In the case under consideration, the bound condition~\eqref{m-bound} for $z''$ can be violated in the coordinate $e(x)$ only. Suppose, for the sake of contradiction, that $z_{e(x)}+1> b(z)-1$. Both $z_{e(x)}$ and $b(z)$ are odd and $z_{e(x)}\leq b(z)$. Therefore $z_{e(x)}=b(z)$.
Recall that  $z$ is ordered and $z_{p}$ is even while $b(z)$ is odd. Thus $z_{p}< b(z)$ which implies $p< e(x)$. 
Since $x_p=z_p$ is even, from the definition of $e(x)$ we get
\[
b(z) = z_{e(x)} \geq z_p = x_p = x_{p+1} = \dots = x_{e(x)}\geq z_{e(x)}.
\]
Thus $z_{e(x)}= z_p$, a contradiction with parities of $z_{e(x)}$ and $z_p$ (the former is odd and the latter is even).

We conclude that  $z''$ is basic,   $b(z'') = b(z)-1$, and $z''\leq x'$. Therefore $\B(x')\geq b(z'') = b(z)-1 = \B(x)-1$.

  \smallskip

  To  prove that  $\B(x') \leq \B(x)-1$, take a basic position $z'$ that supports $x'$.
  
If $b(z')$ is even, then Claim~\ref{ToGeven} implies $\B(x')<\B(x)$ and we are done.

Assume now that $b(z')$ is odd. Let $p$ be the even coordinate in $z'$.

If $p=e(x)$, $z=z'+\mv{e(x)}\preceq x$ and $z$ is basic, since all coordinates are even, and  $|z| = k(b(z')+1)$, and the bound condition~\eqref{m-bound} holds. Thus $\B(x)\geq b(z')+1 = \B(x')+1$.

If $e(x)\ne p$ and  $z'_{e(x)}< x'_{e(x)} = x_{e(x)}$, then  all coordinates are even in $z'' = z'+\mv{p}$, and $|z''| = k(b(z')+1)$, and the bound condition~\eqref{m-bound} holds for $z''$,  and $z''\leq x$. Thus, we get $\B(x)\geq b(z')+1 = \B(x')+1$.

In the remaining case $e(x)\ne p$ and  $z'_{e(x)}= x'_{e(x)} = x_{e(x)}$. It implies  that $x_{e(x)}$ is odd, thus all $x_i$ are odd by definition of M-moves and $e(x)=k+1$.
We have proved already that $\B(x')\geq \B(x)-1$ and $\B(x')\leq \B(x)$ (Claim~\ref{Bnonincrease}). So, to complete the proof, we need to exclude the case $\B(x')=\B(x)$ (provided $x$ is non-exceptional).

Suppose, for the sake of contradiction, that $\B(x')=\B(x)$.   We are going to construct a basic position $z^*$ such that $z^*\preceq x$ and $b(z^*) = b(z')+1$. It would imply that $\B(x)\geq b(z^*)> \B(x')$, a~contradiction.

Note that $z'_i\leq x'_i = x_i-1$ for all $1\leq i\leq k$. It implies that $z'_i< x'_i=x_i-1$ for $i\ne p$ (the former is odd, the latter is even).

If $z'_i\geq x'_i-1 = x_i-2$ for all $i\notin\{p, k+1\}$, then $z'_i = x_i-2$, since $x_i$ is odd. Thus
\[
|x| =  2(k-1) +(x_p-z'_p)+|z'| =(x_p-z'_p)+(k-2) +k(b(z')+1)
\]
(recall that now $z'_{e(x)} = z'_{k+1} = x_{k+1} = x_{e(x)}$). 
If $z'_p = x'_p = x_p-1$ then $x_p-z'_p=1$ and we come to a contradiction with the assumption that $x$ is not exceptional.

So we get a dichotomy: either $x_p-z'_p>1$ or $x_p-z'_p=1$ and there exists $j\notin\{p,k+1\}$ such that $z'_j< x_j-2$.  

In the first case $z^*$ is defined as follows
\[
z^*_i = \left\{
\begin{aligned}
  &z'_p +2 , &&\text{if}\ i=p,\\
  &z'_{k+1} -1 , &&\text{if}\ i=k+1,\\
  &z'_i+1, &&\text{otherwise.}
\end{aligned}\right.
\]
By definition, $z^*\preceq x$, and all $z^*_i$ are even, and
\[|z^*|= kb(z') + 2-1+(k-1) = k(b(z')+1),\] 
and the bound condition holds for $z^*$, since $z'_p< b(z')$ (the former is even, the latter is odd). We come to a contradiction $\B(x)\geq b(z')+1 = \B(x')+1 = \B(x)+1$.

In the second case $z^*$ is defined as follows
\[
z^*_i = \left\{
\begin{aligned}
  &z'_j +3 , &&\text{if}\ i=j,\\
  &z'_{k+1} -1 , &&\text{if}\ i=k+1,\\
  &z'_{p}  , &&\text{if}\ i=p,\\
  &z'_i+1, &&\text{otherwise.}
\end{aligned}\right.
\]
Again,  $z^*\preceq x$, and
all $z^*_i$ are even, and
\[|z^*| = kb(z') + 3-1+(k-2) = k(b(z')+1),\] and the bound condition holds for $z^*$. We come to a contradiction  $\B(x)\geq b(z')+1 = \B(x')+1 = \B(x)+1$.
\end{proof}

\begin{lemma}\label{M=G}
   $\M(x)=\B(x)$ for non-exceptional $x$ and $\M(x)=\B(x)+1$ for exceptional~$x$.
\end{lemma}
\begin{proof}
  Induction on values of $\B(x)$. The base case is Claim~\ref{terminal}.
  Note that there are no M-moves to exceptional positions, since after an M-move at least one coordinate is even. Also, there are no moves between  exceptional positions. Thus the induction step for non-exceptional positions follows from Claims~\ref{FromGeven} and~\ref{FromGodd} depending on parity of $\B(x)$, while for exceptional positions it follows from Claim~\ref{exceptional3}.
\end{proof}

\begin{corollary}\label{M=m}
  $\M(x) = m(x)$ for all $x$.
\end{corollary}

\begin{proof}[Proof of Theorem~\ref{th:critical}.]
From Claim~\ref{exceptional1} we conclude that an exceptional $m$-critical position (the condition (B)) supports only itself. For the rest of positions, the theorem follows from 
  Lemma~\ref{M=G},  since  $m$-critical non-exceptional positions and  basic positions supporting a non-exceptional $x$ are the same.
\end{proof}

\begin{proof}[Proof of Theorem~\ref{th:main}.]
  Due to Lemma~\ref{M=G}, Theorem~\ref{th:critical} and Corollary~\ref{M=m} it remains to prove that $\R(x) = \M(x)$ for all~$x$; in other words, $\R(x) = \B(x)$ for all non-exceptional~$x$ and $\R(x) = \B(x)+1$ for all exceptional~$x$.

  We prove it by induction on the sum of coordinates. In the proof we use an equivalent definition of $\R(x)$. Let $N^+(x)$ be the set of positions such that there exists a move $x\to x'$. 
Then  the remoteness function is defined by the following  recurrence
\begin{equation}\label{e-SM-acyclic-def}
\R(x) =\left\{
\begin{aligned}
&0,  &&\text{ if } \R(N^+(x))=\emptyset,\\
&1+ \min \R(N^+(x))\cap 2\ZZ_{\geq0}\,, && \text{ if } \R(N^+(x))\cap 2\ZZ_{\geq0} \neq \emptyset,\\
&1+ \max \R(N^+(x))\cap 2\ZZ_{\geq0}\,, && \text{ if } \emptyset \neq \R(N^+(x))\subseteq 1+2\ZZ_{\geq0}\,.
\end{aligned}\right. 
\end{equation}

The base case $x=0$ trivially holds.

Assume that $\R(x) = \M(x)$ for all~$x$ such that $\sum_i x_i< q$. Take $y$ such that $\sum_i y_i = q$.

If $y$ is exceptional, then, due to Claim~\ref{exceptional3} and Corollary~\ref{M=m}, for any move $y\to y'$, we have $m(y) = \M(y) = \B(y)+1 = \B(y')+1 = \R(y')+1$ (the last equality is the induction hypothesis) and $\R(y')$ is odd. Therefore $\R(y) = 1+\R(y') = \M(y)$.

If $y$ is $B$-odd and non-exceptional, then $y\to x= y-\mv{e(y)}$ is an M-move to $B$-even position. Due to Claim~\ref{FromGodd}, $\B(x) = \B(y)-1$. Also, $\B(x)\geq \B(y)-2$ for any other move from $y$ due to Claim~\ref{DeltaGodd}. From~\eqref{e-SM-acyclic-def} we conclude that $\R(y) = 1+ (\B(y)-1)= \B(y) $. 

If $y$ is $B$-even, it is  non-exceptional  and there are no moves to $B$-even positions from $y$ due to Lemma~\ref{GevenGeven}. It follows from Claim~\ref{FromGeven} that there is a move to $x$ such that $\B(x) = \B(y)-1$. 
From~\eqref{e-SM-acyclic-def} we conclude that $\R(y) = 1+ (\B(y)-1)= \B(y) $. 
\end{proof}

\section{Structural and algorithmic complexity}

It follows from Theorem~\ref{th:main} that $x$ is a P position in NIM$(k+1,k)$ iff $m(x)=\R(x)$ is even. 

For a fixed $k$, conditions~(A) and (B) in Theorem~\ref{th:critical}  are easily expressed in Presburger arithmetic (see, e.g.~\cite{Hasse18}): they include inequalities and congruences modulo fixed integers. Therefore the predicate $\R(x)=2y$ is also expressed in Presburger arithmetic. It implies that the set of P-positions is semilinear~\cite{Hasse18}, i.e. can be expressed as a~finite union of solutions of systems of linear inequalities combined with
equations modulo some integer (fixed for the set). It was conjectured
in~\cite{GV21} that the set of P-positions of any multidimensional subtraction game with nonnegative vectors of differences is semilinear. The conjecture
was supported by several sporadic examples from~\cite{GHHC20,CGKPV21}. Games NIM$(k+1,k)$ provide an  infinite family of nontrivial subtraction games having this property. 

Thus, for a fixed $k$, there exists a very simple linear time algorithm recognizing  P-positions in  NIM$(k+1,k)$.
Theorem~\ref{th:complexity} gives more. It claims that there exists an
algorithm recognizing P-positions even if $k$ is a part of the input.

It follows from Theorems~\ref{th:critical} and~\ref{th:main} that,
to compute $\R(x)$ efficiently,  it is enough to compute  $\B(x)$ efficiently for non-exceptional positions and compute $m(x)$ for exceptional ones.
Verifying the conditions (B) in Theorem~\ref{th:critical} can be done in polynomial time as well as computing $m(x)$ for exceptional positions.

To compute  $\B(x)$ efficiently for non-exceptional positions,  we use an enhanced version of  Claim~\ref{ordered}.
We need the reverse lexicographical order on $(k+1)$-dimensional integer vectors, which is defined as follows: $x\lex y$ if $x_i=y_i$ for all $i>t$ and $x_t< y_t$.

For even $b$, let $\ell(x,b)$ be the maximal (w.r.t. the reverse lexicographical order)  basic position $z$ such that $z\preceq x$, $b(z) = b$.

\begin{lemma}\label{maxlex}
  For all $x$ and for all even $b$, there exists $1<t\leq k+2$ such that $\ell(x,b)_i = b$ for $i\geq t$ and
  $0\leq x_i-\ell(x,b)_i\leq 1$ for all $1<i<t$.
\end{lemma}

\begin{proof}
  If $i<j$ then decreasing the coordinate $i$ by 2 and increasing the coordinate $j$ by 2 transforms a basic position $z$ to a basic position $z'$ such that $z\lex z'$. This operation is  called  a \emph{shift}.

Note also that re-ordering a vector in the non-increasing order gives a vector that is not smaller than initial one with respect to the reverse lexicographical order.
  
Take now  an ordered basic position  $z$ such that $z\preceq x$, $b(z) = b$ and no shift can be applied to $z$. It means that $z_i\geq \min(x_i-1, b-1)$ for all $1<i\leq k+1$. Since  $b$ is even, then either $z_i\geq x_i-1$ or $z_i=b$.
\end{proof}

Note that $\ell(x,b)$ is uniquely determined by  $t$: there is exactly one even integer among $x_i-1$, $x_i$. It is a key observation for the proof of Theorem~\ref{th:complexity}.

\begin{proof}[Proof of Theorem~\ref{th:complexity}.]
  For exceptional positions $m(x)$ can be found in polynomial time. For the rest of the proof  we assume that $x$  is non-exceptional and our goal is to compute $\B(x)$.
  
  Let $b_{t}(x)$, $1<t\leq k+2$, be the maximal value $b(z)$ for $B$-even basic positions $z$ such that the conditions of Lemma~\ref{maxlex} hold for the specific value of~$t$.
  If there are no $z$ satisfying the requirements, then $b_{t} = -\infty$.
Let $E(x) = \max_{t} b_{t}(x)$.

Due to Lemma~\ref{maxlex},  $E(x)\leq \B(x)$  and equality holds exactly for $B$-even $x$. 
Let $x\to x' = x-\mv{e(x)} $ be an M-move.  If $x$ is $B$-odd  then  $x'$ is $B$-even and $\B(x) = \B(x')+1 = E(x')+1$. Therefore $E(x')=\B(x') = \B(x)-1> E(x)-1$. Similarly, for a $B$-even position $x$ we have $E(x)=\B(x) = \B(x')+1> E(x')+1 $.

It gives the rule to compute $\B(x)$ by $E(x)$:
\[
\B(x) =\left\{\begin{aligned}
  &E(x), &&\text{if}\ E(x)> E(x')+1,\\
  &E(x')+1, &&\text{if}\ E(x)< E(x')+1,
\end{aligned}\right.
\]
where $x\to x'$ is an M-move.
  
To complete the proof, we provide an efficient algorithm that computes $E(x)$. 
For brevity we use the following notation: 
\[
A_t = \sum_{i=2}^{t-1}2\cdot\lfloor x_i/2\rfloor .
\]
To find $b_{t}(x)$, one need to find the maximal $q$ satisfying the requirements
\begin{align}
  &2s + A_{t} + q\cdot 2(k+2-t) = q\cdot 2k,\label{=cond}\\
  &2s\leq x_1,\notag\\
  &2q\leq x_{t},\notag 
\end{align}
and set $b_{t}(x) =2q $. (Here $2s$ is the value of $z_1$.)

Eq.~\eqref{=cond} is equivalent to
\[
q\cdot 2(t-2) = 2s+ A_t\,.
\]
Thus, we have two linear inequalities on  $q$ and the maximal value of $q$ is
\[
\min\left(
\left\lfloor
\frac{2\lfloor x_1/2\rfloor+A_t}{2(t-2)}
\right\rfloor,
\lfloor \frac{x_{t}}2\rfloor
\right)\,.
\]
If $t=2$ the first inequality becomes irrelevant.

Combining all the arguments above, we get the polynomial algorithm to compute $\B(x)$ and $\R(x)$.
\end{proof}

\section*{Acknowledgements:} 
This research was prepared within the framework 
of the HSE University Basic Research Program.  
The fourth  author was supported in part 
by the state assignment topic no.~0063-2016-0003.

\appendix 

\section*{Appendix: P- and N-positions  of  NIM(4,3)}

Here we use notation $[A]$ for the indicator function: $[A]=1$ if $A$
is true and $[A]=0$ otherwise; and  $a \bmod b$ for the residue of $a$
modulo $b$.

\def\labelitemii{$\mathsurround=0pt \circ$}

\medskip

\textbf{Case 0. \boldmath $x_1 + x_2 + x_3 + x_4 = 0 \pmod3$}

\begin{itemize}
    \item If $x_1$ and $x_2$ are both odd then $x$ is an N-position.
    \item Otherwise, if $(x_3 - x_2 - x_1) \geq 0$ then 
    $x$ is a P-position if and only if $(x_1 + x_2)$ is even.
    \item Otherwise, if $(x_3 - x_2 - x_1) = -1$ then 
    $x$ is a P-position if and only if $(x_1 + x_2)$ is odd.
    \item Otherwise, if $(x_1 + x_3 - x_2)$ is even 
    then $x$ is a P-position if and only if $(x_1 + x_2)$ is even.
    \item Otherwise, $x$ is a P-position if and only if 
    $(x_2 + p + q)$ is even, where
\[
\begin{aligned}
  &p = \left[(x_1 + x_3 - x_2) \equiv 3 \pmod 4 \right],\\
  &q = \left[(x_1 + x_2 + x_3 - 2x_4 - 3) = 12k,\  k \in \mathbb{Z}_{\geq 0}\right]. 
\end{aligned}
\]
\end{itemize}

\medskip

\textbf{Case 1. \boldmath $x_1 + x_2 + x_3 + x_4 = 1 \pmod 3$}

\begin{itemize}
    \item If $x_1$ and $x_2$ are both odd then $x$ is an N-position.
    \item Otherwise, if $(x_3 - x_2 - x_1) \geq 0$, or 
    $(x_3 - x_2 - x_1)$ is even, or $(x_3 - x_2 - x_1) = -3 $
    then $x$ is a P-position if and only if $(x_1 + x_2)$ is even.
    \item Otherwise, if $(x_3 - x_2 - x_1) = -1$ or $(x_3 - x_2 - x_1) = -5 $
    then $x$ is a P-position if and only if $(x_1 + x_2)$ is odd.
    \item Otherwise, $x$ is a P-position if and only if 
    $(x_2 + p + q)$ is odd, where
\[
\begin{aligned}
  & p = \left[(x_1 + x_3 - x_2) \equiv 1 \pmod 4 \right],\\
  & q = \left[(x_1 + x_2 + x_3 - 2x_4 - 7) = 12k,\ k \in \mathbb{Z}_{\geq 0} \right].
\end{aligned}
\]
\end{itemize}

\medskip

\textbf{Case 2. \boldmath $x_1 + x_2 + x_3 + x_4 = 2 \pmod 3$}

\begin{itemize}
    \item Let $x_1$ and $x_2$ be both odd:
    \begin{itemize}
        \item If $(x_3 - x_2 - x_1) \geq 0$ or 
        $(x_3 - x_2 - x_1) \in \{-1, -3, -4, -7\}$ 
        then $x$ is an N-position.
        \item Otherwise, $x$ is a P-position if and only if $p$ is
          odd, where
\[
p =
\left[
  \begin{aligned}
  &(x_1 + x_2 + x_3 - 2x_4 - c) = 12k, \ k \in \mathbb{Z}_{\geq 0},
    \\
    &c =2+ 3\big((x_1 + x_2) \bmod 4\big)
  \end{aligned}
  \right].
\]
    \end{itemize}
    \item Let $x_1$ be odd and $x_2$  even.
    \begin{itemize}
        \item If $(x_3 - x_2 - x_1) \geq 0$ or $(x_3 - x_2 - x_1) \in \{-2, -3, -6\}$ 
        then $x$ is an N-position.
        \item Otherwise, if $(x_3 - x_2 - x_1) = -1$ then $x$ is an P-position.
        \item Otherwise, $x$ is a P-position if and only if $(p + q)$
          is odd, where
          \[
          \begin{aligned}            
&p = \left[    (x_1 + x_3 - x_2) \equiv 1 \pmod 4  \right],\\
            &q = \left[
              \begin{aligned}
                &(x_1 + x_2 + x_3 - 2x_4 - c) = 12k, && k \in
                \mathbb{Z}_{\geq 0},\\
                & c = 5, &&\text{if  $(x_1 + x_2)$ is odd},\\
                & c = (8 - 3((x_1 + x_2) \bmod 4)), &&\text{otherwise.}
  \end{aligned}
              \right]            
          \end{aligned}
\]
    \end{itemize}
    \item Let $x_1$ be even.
    \begin{itemize}
        \item If $(x_3 - x_2 - x_1) \geq 0$ or $(x_3 - x_2 - x_1) \in \{-2, -3, -6\}$ 
        then $x$ is a P-position if and only if $x_2$ is even.
        \item Otherwise if $(x_3 - x_2 - x_1) = -1 $
        then $x$ is a P-position if and only if $x_2 $ is odd.
        \item Otherwise $x$ is a P-position if and only if $(p + q +
          x_2)$ is even, where
\[
\begin{aligned}
  &p = \left[(x_1 + x_3 - x_2) \equiv 3 \pmod 4\right],\\
  &q = \left[
    \begin{aligned}
     & (x_1 + x_2 + x_3 - 2x_4 - c) = 12k, && k \in \mathbb{Z}_{\geq
        0},\\
      & c = 5, && \text{if $(x_1 + x_2)$ is odd},\\
      & c = 2 + 3\big((x_1 + x_2) \bmod 4\big), &&\text{otherwise.}
    \end{aligned}
    \right]
\end{aligned}
\]
    \end{itemize}
\end{itemize}

\end{document}